\documentclass[11pt]{amsart}
\usepackage[linktocpage]{hyperref}
\usepackage{amssymb, paralist, xspace, graphicx, url, amscd, euscript, mathrsfs,stmaryrd,epic,eepic,color}
\usepackage[all]{xy}
\usepackage{amsthm}
\SelectTips{cm}{}

\numberwithin{equation}{section}

1
\numberwithin{subsection}{section}

\allowdisplaybreaks[1]

\newenvironment{enumeratea}
{\begin{enumerate}[\upshape(a)]
 \setlength{\itemsep}{0pt}
  \setlength{\parskip}{0pt}
  \setlength{\parsep}{0pt}
  }
{\end{enumerate}}

\newenvironment{enumerate1}
{\begin{enumerate}[\upshape (1)]}
{\end{enumerate}}

\newtheorem*{namedtheorem}{\theoremname}
\newcommand{\theoremname}{testing}

\newtheorem{theorem}{Theorem}
\newtheorem{proposition}{Proposition}
\newtheorem{proposition-definition}[subsection]
{Proposition-Definition}

\newtheorem{lemma}{Lemma}
\theoremstyle{definition}

\newtheorem{remark}{Remark}

\newtheorem*{definition*}{Definition}
\theoremstyle{remark}

\newcommand\cC{\mathcal{C}}
\newcommand\cD{\mathcal{D}}
\newcommand\cE{\mathcal{E}}

\newcommand\cH{\mathcal{H}}

\newcommand\cK{\mathcal{K}}

\newcommand\cM{\mathcal{M}}

\newcommand\cO{\mathcal{O}}
\newcommand\cP{\mathcal{P}}

\newcommand\cU{\mathcal{U}}

\newcommand\cY{\mathcal{Y}}

\newcommand\CC{\mathbb{C}}

\newcommand\HH{\mathbb{H}}

\newcommand\NN{\mathbb{N}}

\newcommand\PP{\mathbb{P}}
\newcommand\QQ{\mathbb{Q}}

\newcommand\ZZ{\mathbb{Z}}

\newcommand\fX{\mathfrak{X}}

\newcommand\arr{\ifinner\to\else\longrightarrow\fi}

\def\displaytimes_#1{\mathrel{\mathop{\times}\limits_{#1}}}

\def\displayotimes_#1{\mathrel{\mathop{\bigotimes}\limits_{#1}}}

\newdir{ >}{{}*!/-5pt/@{>}}

\newcommand\doublelong[2]{\mathbin{\xymatrix{{}\ar@<3pt>[r]^{#1}
\ar@<-3pt>[r]_{#2}&}}}

\newlength{\ignora}

\newcommand{\MT}{\left(\begin{array}{cc}1 & 1 \\0 & 1\end{array}\right)}
\newcommand{\MS}{\left(\begin{array}{cc}0 & -1 \\1 & 0\end{array}\right)}

\newcommand{\Mp}{{Mp_2}}
\newcommand{\Pic}{\rm{Pic}}
\newcommand{\Mod}{{\rm Mod}}

\theoremstyle{plain}
\theoremstyle{definition}

\begin{document}

\title{Modular forms and special cubic fourfolds}

\author{Zhiyuan Li,~~ Letao Zhang}

\address{Department of Mathematics\\
Stanford University\\
Building 380\\
Stanford, CA 94305\\
U.S.A.}
\address{Department of Mathematics\\
Rice University\\
6100 Main Street\\
Houston, TX 77005\\
U.S.A.}

\email{zli2@stanford.edu, letao.zhang@rice.edu}

\begin{abstract}
 We study the degrees of special cubic divisors on moduli space of cubic fourfolds with at worst ADE singularities.
 In this paper, we show that the generating series of the degrees of such divisors is a level three modular
 form.

\end{abstract}

\maketitle

%
%
%
%

\begin{section}{Introduction}

The classical Noether-Lefschetz locus for degree $d$ hypersurfaces
in $\PP^3$ is the locus in the Hilbert scheme space
$\PP^{{{d+3}\choose{3}}-1}$  where the Picard rank is greater than
one. For $d\geq4$, the Noether-Lefschetz loci are known to be a
countable union of proper subvarieties of
$\PP^{{{d+3}\choose{3}}-1}$ by Griffiths and Harris \cite{GH85}. A
natural question is to find the degrees of these subvarieties. For
quartic surfaces in $\PP^3$, Maulik and Pandharipande \cite{MP07}
showed that the Noether-Lefschetz loci are divisors and the degrees of these divisors are the Fourier
coefficients of certain modular forms.

In higher dimensional cases, cubic fourfolds have received a lot of
attention since their period map behaves quite nicely. Specifically,
the period domain for cubic fourfolds is a bounded symmetric domain
of type IV and the global Torelli theorem holds
(cf.~\cite{Vo86,Vo09}).

The analogues of Noether-Lefschetz loci for surfaces are the loci of
special cubic fourfolds studied by Hassett \cite{Ha00}. A smooth cubic fourfold $X$ in $\PP^5$ is {\it special}
of discriminant $d>6$ if it contains an algebraic surface $S$, and
the discriminant of the saturated lattice spanned by $h^2$ and $[S]$
in $H^4(X,\ZZ)$ is $d$, where $h=c_1(\cO_X(1))$.
 The Zariski closure of the collection of such cubic fourfolds forms an irreducible divisor $\cC_d$ in the moduli space $\cM$
 (cf.~\cite{La09}) of cubic fourfolds with at worst isolated ADE singularities and it is nonempty if and only if $d\equiv0,2\mod 6$.
 We interpret $\cC_6$ as the set of singular cubics in $\cM$.
The special cubic divisor $C_d$ in the Hilbert scheme $\PP^{55}$ of cubic hypersurfaces is the lift of $\cC_d$ (see $\S$2 for more details).
In the present paper, we study the degree of special cubic divisors $C_d$. Our main result is:
\begin{theorem}\label{thm:1}
Let $\Theta(q)=-2+\sum\limits_{d>2}^{\infty} \deg(C_d)q^{\frac{d}{6}}$ be the generating series for the degrees of the
special cubic divisors. Then $\Theta(q)$ is a modular form of weight $11$ and level $3$ with expansion:
\begin{align*}
\Theta(q)= &-\alpha^{11}(q)+162\alpha^8(q)\beta(q)+91854\alpha^5(q)\beta^2(q)+ 2204496\alpha^2(q)\beta^3(q)\\
      &-\alpha^{11}(q^{\frac{1}{3}})+66\alpha^8(q^{\frac{1}{3}})\beta(q^{\frac{1}{3}})-1386\alpha^5(q^{\frac{1}{3}})\beta^2(q^{\frac{1}{3}})+9072\alpha^2(q^{\frac{1}{3}})\beta^3(q^{\frac{1}{3}})\\
          =&-2+192 q+3402 q^{\frac{4}{3}}+196272q^2+915678q^{\frac{7}{3}}+\ldots
\end{align*}

where
\begin{equation}\label{eq1.2}
\alpha(q)=1+6\sum\limits_{n\geq1} q^n\sum\limits_{d|n}\left(\frac{d}{3}\right) ~\textrm{and}~~ \beta(q)=\sum\limits_{n\geq1} q^n\sum\limits_{d|n}(n/d)^2\left(\frac{d}{3}\right)\end{equation}
are level three modular forms of weight 1 and 3 that generate the space of modular forms with respect to the group $\Gamma_0(3)$ (see~$\S$3.2).
Here, $\left(\frac{d}{3}\right) $ denotes the Legendre symbol.
\end{theorem}
The approach to Theorem 1 is via the result of Borcherds \cite{Bo99} and Kudla-Milson \cite{KM90}. The degrees of $C_d$ are the Fourier coefficients of a vector-valued modular form. As in \cite{MP07}, the Noether-Lefschetz numbers are related to the reduced Gromov-Witten (GW) invariants of K3 surfaces.  We hope there is a similar GW-theory interpretation of $\deg(C_d)$.

\subsection*{Outline of paper} In section 2, we review some classical result on cubic fourfolds and describe the special cubic divisors
from an arithmetic perspective. Section 3 is the central section of this paper. We recap Borcherds' work on Heegner divisors to prove the modularity of a
 vector-valued generating series of $\deg(C_d)$. This vector-valued modular form can be expressed explicitly in terms of some well-known modular forms. The proof of our main theorem is presented in the last section.
\\

 After posting our paper on the arXiv server, we learned from Atanas Iliev that, in forthcoming work,  Atanas Iliev, Emanuel Scheidegger,
 and Ludmil Katzarkov in \cite{AEL} have independently proved Theorem \ref{thm1} using a different basis of the space of vector-valued modular forms.\\

 {\bf Acknowledgements}.  The first author was supported by NSF grant 0901645. The authors are grateful to their advisor Brendan Hassett for introducing
  this problem, and many useful discussions. We have benefited from discussions with Radu Laza.
  We would also like to thank the referee for many helpful comments.
\end{section}
%
%
%
%
\begin{section}{Special Cubic fourfolds and Heegner divisors}\label{sec:2}

In this section, we review some results on special cubic divisors and the relation with Heegner divisors associated to a signature $(m,2)$ lattice, which is defined from an arithmetic perspective. Throughout the paper,
 we denote by $L^\vee$ the dual of a lattice $L$ and $O(L)$ its associated orthogonal group.
\begin{subsection}{Period domain}\label{sec:2.1}
Let $X$ be a smooth cubic fourfold in $\PP^5$. Denote by $\Lambda$ the middle cohomology group $H^4(X,\ZZ)$ containing $h^2$.
It is well known (e.g.~\cite{Ha00},\cite{Vo09}) that the primitive middle
cohomology $H^4(X,\ZZ)_{prim}$ is isometric to the lattice
 \begin{equation}\label{eq2.1}
 \Lambda_0:= A_2\oplus U^{\oplus2} \oplus E_8^{\oplus 2}
 \end{equation}
under the intersection form $\left<,\right>$, with an associated period domain $\cD$ that is a connected component of
$$\cD^{\pm}:=\{\omega\in \PP(\Lambda_0 \otimes_\ZZ \CC)| \left<\omega,\omega \right>=0,\left<\omega, \bar{\omega}\right> < 0\}.$$
Here, $A_2$ and $ E_8$  are root lattices corresponding to the root
systems of the same names,  and $U$ is the hyperbolic lattice of rank
two. The monodromy group $\Gamma\subset O^{+}(\Lambda_0)$ (i.e. the
identity connected component of $O(\Lambda_0)$) is defined by
$$\Gamma=\{g\in O^{+}(\Lambda_0)|~g~\hbox{acts trivially on
$\Lambda_0^\vee/\Lambda_0$}\}$$ acts on $\cD$ and the arithmetic
quotient $\Gamma\backslash\cD$ is a quasi-projective variety
parametrizing the periods of cubics.

Hassett \cite{Ha00} has defined irreducible divisors $D_d\subseteq
\Gamma\backslash\cD$ as follows:
 \begin{definition*}Let $L$ be a rank-two positive definite saturated sublattice of $\Lambda$ containing $h^2$ of
  discriminant $d$. There is an associated hyperplane
\begin{align*}\cH_L:=\{\omega\in \cD~|~\omega\perp L\}\end{align*} in $\cD$.  Then $D_d$ is defined as the quotient by $\Gamma$ of the union of all such hyperplanes $\cH_L$.  \end{definition*}

On the other hand,  as constructed by Laza \cite{La09} via geometric invariant theory (GIT),  the moduli space $\cM$ of cubic fourfolds with at worst isolated ADE singularities
 is a Zariski open subset of $\PP(W)^s/\!\!/SL_6(\CC)$, where $W=H^0(\PP^5,\cO_{\PP^5}(3))$ and $\PP(W)^s$ denotes the GIT stable points of $\PP(W)\cong \PP^{55}$ under
 the action of $SL_6(\CC)$. Together with Voisin's Global Torelli theorem (cf.~\cite{Vo86,Vo09}), Laza \cite{La10} and also Looijenga \cite{Lo09} have shown that there is an extended  period map
\begin{equation}\label{eq2.2}
\cP:\cM\rightarrow \Gamma\backslash\cD,
\end{equation} which is an open immersion. The complement of  the image of $\cP$ in $\Gamma\backslash\cD$ is
 $D_2$ corresponding to degenerations of determinant cubic hypersurfaces (cf.~\cite{Ha00} $\S$4.4).
 Moreover, $\cC_d\subset\cM$ is exactly the pullback of $D_d$ via $\cP$.

Let $\varphi:\PP(W)\dashrightarrow \cM$ be the natural quotient map;
then the special cubic divisor $C_d\subset \PP(W)$ is the Zariski
closure of the pullback of $\cC_d$ via $\varphi$. \end{subsection}

\subsection{Heegner divisors}
In general, let $M$ be an even lattice of signature $(m,2)$ with an associated domain $\cD_M$ as a connected component of
 $$\cD_M^{\pm}:=\{ \omega\in \PP(M\otimes \CC)|\left<\omega,\omega\right>=0, \left<\omega,\bar{\omega}\right><0\}; $$
  there is an arithmetic group $$\Gamma_M =\{ g\in O^{+}(M)|~g~\hbox{ acts trivially on $M^\vee/M$}\}$$
acting on $\cD_M$. For $n\in\QQ^{+}$ and $\gamma\in M^\vee$,  the
{\it Heegner  divisor} \cite{Br02} $y_{n,\gamma}$  on
$\Gamma_M\backslash \cD_M$ is defined by
 \begin{equation}\label{eq2.3}y_{n,\gamma}= \Gamma_M\backslash\left(\sum_{\frac{1}{2}\left<v,v\right>=n,\textrm{ }v\equiv\gamma\textrm{ mod } M} v^{\bot}\right),
 \end{equation}
where $v^\bot=\{w \in \cD_M|~\left<w,v\right>=0\}$  is a hyperplane of $\cD_M$.
When $n=0$, we take $y_{0,0}$ to be the $\QQ$-Cartier divisor coming from  $\cO(1)$ on $\cD_M\subseteq \PP(M\otimes\CC)$.
Similarly to \cite{MP07}, it is clear from definition that $\cP^\ast[y_{0,0}]$ is actually the Hodge bundle $R^3f_\ast(\Omega^1_{\cY/\cM})$ on $\cM$,  where $f:\cY\rightarrow\cM$ is the universal family and $\Omega^1_{\cY/\cM}$ is the relative sheaf of holomorphic 1-forms.

\begin{remark}
One can see that $y_{n,\gamma}$ is the arithmetic quotient of a Hermitian symmetric subdomain of $\cD_M$. They are called {\it special cycle} on $\Gamma_M\backslash\cD_M$ in Kudla's program \cite{Ku97}.
\end{remark}
\noindent Taking $M=\Lambda_0$, we have $\cD_{\Lambda_0}=\cD$ and
$\Gamma_{\Lambda_0}=\Gamma$. Note that
$\Lambda_0^\vee/\Lambda_0\cong \ZZ/3\ZZ$, one can choose
representatives $\gamma_i\in \Lambda_0^\vee/\Lambda_0$
 with $ \frac{1}{2}\left<\gamma_i,\gamma_i\right>\equiv\frac{i^2}{3}\mod \ZZ $ for $i=0,1,2.$
\begin{lemma}\label{lem1}
 The  Heegner divisors $y_{n,\gamma}=y_{n,-\gamma}=D_{d}$ on
  $\Gamma\backslash\cD$, where
   $$n=\frac{d}{6}~~~\hbox{and} ~~~\gamma\equiv \frac{d}{2}\gamma_1 \mod \Lambda_0,~
   \textrm{for}~ (n,\gamma)\neq
   (0,0).$$
\begin{proof} The redundancy
 $y_{n,\gamma} =y_{n,-\gamma}$ is because of  the symmetry
 $\left<v\right>^{\perp}=\left<-v\right>^{\perp}$. Let $L\subseteq \Lambda$ be a rank 2 negative sublattice containing $h^2$ and of discriminant $d$.
 Assume that $L$ is generated by $h^2$ and $\zeta$. Then there is a  bijection between the two sets of hyperplanes as follows:
  \begin{align*}\{\cH_{L}\} &\longleftrightarrow\{ v^{\perp}\}\\
   \zeta&\longleftrightarrow v=\zeta+\frac{\left<\zeta,h^2\right>}{3}h^2,\end{align*}
since one can verify  that
\begin{align*} \frac{1}{2}\left<\zeta+\frac{\left<\zeta,h^2\right>}{3}h^2, \zeta+\frac{\left<\zeta,h^2\right> }{3}h^2\right>=n \\ \zeta+\frac{\left<\zeta,h^2\right>}{3} h^2\equiv  \pm \frac{d}{2}\gamma_1 \mod  \Lambda_0.\end{align*}
\end{proof}
\end{lemma}

As an application, we show the following result:

\begin{proposition} The Picard group $ \Pic_\QQ(\Gamma\backslash\cD)$ has rank two and  is spanned by Heegner divisors $y_{0,0}$ and $y_{1/3,\gamma_1}=D_2$.
\end{proposition}

\begin{proof} Let $\Pic_\QQ(\Gamma\backslash\cD)^{Heegner}$ be the subgroup of $\hbox{Pic}_\QQ(\Gamma\backslash\cD)$ generated by the Heegner divisors. By applying the general formula from Bruinier \cite{Br02} $\S$5.2 to the lattice $\Lambda_0$, we get
$$\dim \rm{Pic}_\QQ(\Gamma\backslash\cD)^{Heegner}=2.$$

On the other hand,  the moduli space $\cC$ of cubic fourfolds with at worst isolated ADE singularities
 is an open subset of $\Gamma\backslash\cD$ via the extended period map, and the complement  of $\cC$ is the irreducible Heegner divisor $D_2$.
If the Picard number of $\cC$ is at most one, then $\dim\Pic_\QQ(\Gamma\backslash\cD) $ is at most two and $\Pic_\QQ(\Gamma\backslash\cD)$ has to be the same as $\Pic_\QQ(\Gamma\backslash\cD)^{Heegner}$ by dimension considerations.

We prove that the dimension of  $\Pic_\QQ(\cC)$ is at most one.
 Observe that $\cC$ is constructed via the GIT quotient $\cU/\!\!/SL_6(\CC)$,
where $\cU\subset \PP^{55}$ is the open subset of the Hilbert scheme parameterizing all cubic hypersurfaces with at worst ADE singularities.
Then $\Pic(\cU)\cong \Pic(\PP^{55})$ has rank one  since the boundary of $\cU$ in $\PP^{55}$ has codimension at least two.

Let $\Pic(\cU)_{SL_6(\CC)}$ be the set of $SL_6(\CC)$-linearized line bundles on $\cU$. There is an injection $$\Pic(\cU/\!\!/SL_6(\CC))\hookrightarrow \Pic(\cU)_{SL_6(\CC)}$$ by \cite[Proposition 4.2.]{KKV89}.
Our assertion follows from the fact  the forgetful map $\Pic(\cU)_{SL_6(\CC)}\rightarrow \Pic(\cU) $ is an injection.

\end{proof}
\begin{remark}
The moduli space of quasi-polarized K3 surface $\cK_g$ is a 19-dimensional locally Hermitian symmetric variety associated to $SO(19,2)$. It is conjectured by Maulik and Pandharipande
 that the Picard group of $\cK_g$ is rationally spanned by Noether-Lefschetz divisors.  This conjecture has been verified for low degree $K3$ surfaces (cf.~ \cite{Sh80}, \cite{Sh81}, \cite{LT13}). We also refer the readers to  \cite{HH07} and \cite{BMM11} for some recent results in this subject.
\end{remark}

\subsection{The degree of special cubic divisors}\label{sbsec2.3}
Analogous to Noether-Lefschetz numbers of K3 surfaces \cite{MP07},
the degree of $C_d$ can be computed via intersection with a test
curve. Let $\pi:\fX\rightarrow \PP^1$ be a Lefschetz pencil of cubic
hypersurfaces in $\PP^5$. It yields a natural morphism
$$\iota_\pi:\PP^1\rightarrow \cM~,$$ which factors through the rational map
$\varphi:\PP(W)\dashrightarrow\cM$. It is not difficult to see that
$\deg(C_d)$ is the same as the intersection number $
\int_{\PP^1}\iota_\pi^\ast [\cC_d]$.

Let $\kappa_\pi:\PP^1\rightarrow \Gamma\backslash\cD$ be the
composition of $\iota_\pi$ and $\cP$. If we set
\begin{equation}\label{eq221}N_d=\int_{\PP^1} \kappa_\pi^\ast [D_d],\end{equation}
then we have  $N_d=\deg(C_d)$ for $d>2$, and $N_2=0$ since there are no {\it determinantal cubic fourfolds} in
a Lefschetz pencil $\pi:\fX\rightarrow \PP^1$. The generating series $\Theta(q)$ can be rewritten as
\begin{equation}\label{eq2.5}\Theta(q)=-2+\sum\limits_{d>0}N_dq^{\frac{d}{6}}.
\end{equation}

\subsection*{Some examples} One can see that there is a natural enumerate geometry interpretation of  $\deg(C_d)$,  and some of them can be computed using geometric methods when $d$ is small.
\begin{enumerate1}
 \item The degree of $C_6$ counts the number of  singular fibers in $\fX$.  The {\it first jet bundle} $J^1(\cO_{\PP^5}(3))$ \cite{Sa89}  of $\cO_{\PP^5}(3)$ parametrizes
 all nodal cubic hypersurfaces in $\PP^5$. Then $\deg(C_6)$ equals to the top Chern class of $J^1(\cO_{\PP^5}(3))$, which is $192$.
 \item The degree of $C_8$ counts the number of planes contained in the fibers of $\fX$. Let $Gr(3,6)$ be the Grassmannian parametrizing all planes in $\PP^5$. The planes contained in a cubic
 hypersurface of $\PP^5$ are parametrized by certain vector bundle $\cE$ on $Gr(3,6)$. Via  standard Schubert calculus,  one can show $\deg(\cC_8)$ equals $3402$, which is the top chern class of  $\cE$.
 \item The degree of $C_{14}$ counts the number of  Pfaffian cubic fourfolds (cf.~\cite{Ha00}) in $\fX$, which equals to $915678$ according to our main theorem.
\end{enumerate1}
\end{section}
\begin{section}{Modular forms associated to signature $(20,2)$ lattices}\label{sec:3}
In this section, we introduce the vector-valued modular form associated to an even lattice and prove the modularity of the  generating
series of special cubic divisors from Borcherds' work \cite{Bo99}.

\subsection{Vector valued modular forms}\label{sec:3.1}

The {\it metaplectic} double cover  $\Mp(\ZZ)$ of $SL_2(\ZZ)$ consists of
pairs $\left(A, \phi(\tau)\right)$,  where $$ A=
\left(\begin{array}{cc}a & b \\c & d\end{array}\right)\in SL_2(\ZZ),
~~\phi(\tau)=\pm\sqrt{c\tau+d}. $$ It is well-known that $\Mp(\ZZ)$
is generated by
$$T=\left(\MT, 1\right) ,~~~S=\left(\MS, \sqrt{\tau}\right) .$$

Let $\HH$ be the  complex upper half-plane. Suppose $\rho$ is a
representation of $\Mp(\ZZ)$ on a finite dimensional complex vector
space $V$, such that $\rho$ factors through a finite quotient. For
any $k\in \frac{1}{2} \ZZ$,  a vector-valued modular form $f(\tau)$
of weight $k$ and type $\rho$ on $V$ is a  holomorphic function on
$\HH$, such that
 $$f(A\tau)=\phi(\tau)^{2k}\cdot \rho(g)
(f(\tau)),~\textrm{for all} ~g=(A,\phi(\tau))\in \Mp (\ZZ)~.$$

\noindent When $\dim V=1$, this recovers the definition of
scalar-valued modular forms with a character.

Given a lattice M of signature $(b^+,b^-)$ with a bilinear form $\left<,\right>$, there is a Weil
representation $\rho_M$ of $\Mp(\ZZ)$ on the group ring $\CC[M^\vee/M]$ defined by the action of the generators as follows:
\begin{align*}
\rho_M(T)v_{\gamma}&=e^{2\pi i\frac{\left<\gamma,\gamma\right>}{2}}v_\gamma
\\
\rho_M(S)v_{\gamma}&=\frac{\sqrt{i}^{b^--b^+}}{\sqrt{|M^\vee/M|}}\sum\limits_{\delta\in M^\vee/M} e^{-2\pi i\left<\gamma,\delta\right>}v_{\delta},\end{align*}
where $v_\gamma$ is the standard basis of $\CC[M^\vee/M]$ for $\gamma\in M^\vee/M$. We denote by $\Mod(\Mp(\ZZ),k,\rho_M)$ the space of modular forms of weight $k$ and type $\rho_M$.

Now we take $M=\Lambda_0$ and denote by $v_i$ the standard basis of
 $\CC[\Lambda_0^\vee/\Lambda_0]$ corresponding to element $\gamma_i\in \Lambda_0^\vee/\Lambda_0$ as in $\S$ 2.2. Our first result is:
\begin{theorem}
Let $\overrightarrow{\Theta}(q)$ be the vector-valued generating
series of $N_d$ (\ref{eq221}) defined by
\begin{equation}\label{eq3.1} \overrightarrow{\Theta}(q):=\deg(R^3\pi_\ast(\Omega^1_{\fX/\PP^1}))v_0+\sum\limits_{i=0}^2\sum\limits_{d\equiv i^2\atop~mod~ 3}^{\infty} N_d q^{\frac{d}{6}}v_i. \end{equation}
 Then $\overrightarrow{\Theta}(q)$ is an element of $\Mod(\Mp(\ZZ), 11,\rho_{\Lambda_0})$.
\end{theorem}

\begin{proof} In general, as shown in \cite[Theorem 4.5]{Bo99} and \cite[Theorem 5.6]{Mc03}, the generating series for
Heegner divisors associated to a lattice $M$ of signature $(m,2)$
\begin{equation}\label{eq3.2}\overrightarrow{ \Phi}_M(q):=\sum\limits_{n\in \QQ \geq 0}\sum\limits_{\gamma\in M^\vee/M } y_{n,\gamma}q^n v_\gamma\end{equation}
is an element in $ \Pic(\Gamma_M\backslash\cD_M)\otimes_\ZZ \Mod(\Mp(\ZZ), 1+\frac{m}{2}, \rho_M)$.

In our situation $M=\Lambda_0$, we have $y_{n,\gamma}=D_{d}$ by
Lemma \ref{lem1} and thus the generating series
\begin{equation}\label{eq3.3}\overrightarrow{ \Phi}_{\Lambda_0}(q)= y_{0,0}v_0+\sum\limits_{i=0}^2\sum\limits_{d\equiv i^2\atop~mod~ 3}^{\infty} D_d q^{\frac{d}{6}}v_i
\end{equation}
is a vector-valued modular form of weight $11$ and type $\rho_{\Lambda_0}$ with coefficients in $\Pic(\Gamma\backslash\cD)$.

Next, let $\lambda\in \Pic(\Gamma\backslash\cD)^\ast$ be a linear function defined by
$$\lambda(E)=\int_{\PP^1} \kappa_\pi^\ast[E], ~\forall E\in\Pic(\Gamma\backslash\cD).$$
Then as shown in $\S$2.2, we have $\lambda(D_d)=N_d$ and $$\lambda(y_{0,0})=\int \kappa_\pi^\ast [y_{0,0}]=\int_{\PP^1} R^3\pi_\ast(\Omega^1_{\fX/\PP^1}),$$
which is the degree of
the Hodge bundle $R^3\pi_\ast(\Omega^1_{\fX/\PP^1})$.
It follows that $\overrightarrow{\Theta}(q)=\lambda\otimes\overrightarrow{\Phi}_{\Lambda_0}(q)$ is an element in $\Mod(\Mp(\ZZ), 11,\rho_{\Lambda_0})$.
\end{proof}
\begin{remark}
In Borcherds' setting, the lattice $M$ has signature $(2,m)$ and the generating series of Heegner divisors are vector-valued modular forms of type $\rho_M^\ast$ (dual of $\rho_M$).
 For $M$ with signature $(m,2)$, one can get \eqref{eq3.2} by transferring the lattice to $-M$, which has signature $(2,m)$ and $\rho_{-M}^\ast\cong\rho_M$.
\end{remark}

\subsection{Construction of modular forms}Here we introduce some modular forms which will be used later.
\subsubsection{Scalar-valued Eisenstein series}
 The classical Eisenstein series
 \begin{equation}\label{eq5.1}E_{k}(q)=1-\frac{2k}{B_k} \sum\limits_{n\geq 1}\sum\limits_{d|n} d^{k-1}q^n\end{equation}
 is a modular form of weight $k$ for $SL_2(\ZZ)$ for $k=2l>2$  , where $B_k$ is  the Bernoulli number.

Let us denote by $\Gamma_0(3)$ (resp. $\Gamma^0(3)$) the arithmetic
subgroups in $SL_2(\ZZ)$ defined by
$$\left\{
\left(\begin{array}{cc}a & b
\\c & d\end{array}\right)\in SL_2(\ZZ)~|~c\equiv 0\mod 3~\right\}
\hbox{(resp. $b\equiv 0\mod3$)},$$
and let $\chi:SL_2(\ZZ)\rightarrow \{0,\pm1\}$ be the nontrivial Dirichlet character  modulo 3 on $SL_2(\ZZ)$, i.e.
 \begin{equation}\label{eq5.2}\chi\left(\begin{array}{cc}a & b \\c & d\end{array}\right)=\chi_{-3}(d),~~~\forall  \left(\begin{array}{cc}a & b \\c & d\end{array}\right)\in SL_2(\ZZ)
 \end{equation} where $\chi_{-3}:\ZZ\rightarrow \{0,\pm1\}$ is the nontrivial Dirichlet character modulo 3.

\begin{proposition}\label{prop1}Assume that $k>0$ is an odd integer. Then the Eisenstein series
\begin{equation}\label{eq422} E_k(q,\chi):=\begin{cases} 1+6\sum\limits_{n\geq 1}\sum\limits_{d|n} \chi_{-3} (\frac{n}{d})q^n & k= 1,\\
    \sum\limits_{n\geq 1}\sum\limits_{d|n}d^{k-1}\chi_{-3}(\frac{n}{d}) q^n  & k\geq3.
\end{cases}
\end{equation}
is a modular form of weight $k$ with character $\chi$ for congruence group $\Gamma_0(3)$. Moreover,  $E_1(q,\chi)=\alpha(q)$ and $E_3(q,\chi)=\beta(q)$.

Respectively, $E_k'(q,\chi)=E_k(q^{\frac{1}{3}},\chi)$ is a modular form of weight $k$ with character $\chi$ for congruence group $\Gamma^0(3)$.
\end{proposition}
\begin{proof}
See \cite[Lemma 10.2 \&10.3]{Bo00} for the modularity of \eqref{eq422}.
Since the Legendre symbol of $3$ and the Dirichlet character $\chi_{-3}$ coincide, then $E_1(q,\chi)=\alpha(q)$ and $E_3(q,\chi)=\beta(q)$, where $\alpha$ and $\beta$ are as in Theorem 1.

The modularity of $E_k'(\tau,\chi)$ comes from
\cite[Theorem 4.2.3]{DS05}  and \cite{DS05} $\S$4.8.
\end{proof}

\subsubsection{Vector-valued Eisenstein series}\label{sbsec6.2}
Let $k\in\frac{1}{2}\ZZ $ and $M$ an even lattice. The vector-valued Eisenstein series $\overrightarrow{E}_k(q)$ on $\CC[M^\vee/M]$ constructed
 via Petersson slash operator \cite{Br02} are vector-valued modular forms of weight $k$ and type $\rho_M$. The equivalence of Weil representations $\rho_{\Lambda_0}$ and  $\rho_{A_2}$ implies
\begin{equation}\label{eq:3.3}
\Mod(\Mp(\ZZ), k,\rho_{\Lambda_0})=\Mod(\Mp(\ZZ), k,\rho_{A_2}).
\end{equation}
Let $\overrightarrow{E}_k(q)$ be the vector-valued Eisenstein series associated to $A_2$ of weight $k$. For $k>2$, it is given by \cite{BK01} that
 \begin{align*}
\overrightarrow{E_{k}}(q)=2v_0+\sum\limits_{\gamma\in  W^\vee/W}~\sum\limits_{n\in \ZZ-\frac{1}{2}\gamma^2\atop n> 0} \frac{2^{k+1}\pi^kn^{k-1}(-1)^{(k-1)/2}}{\sqrt{3} \Gamma(k) L(k,\chi_{-3})}
  \prod\limits_{p|18n} \frac{L_{\gamma,n}(k,p)}{1-\chi_{-3}(p)p^{-k}}q^n v_\gamma.
\end{align*}
Here,  $L(k,\chi_{-3})$ denotes the Dirichlet L-series with character $\chi_{-3}$ and $L_{\gamma, n}(k,p)$ is the local Euler product defined as following:
\begin{align*}
d_\gamma&=\hbox{min}\{ b\in \NN,~ b\gamma\in W'\}; \\  \omega_p&=1+2v_p(2d_\gamma n),~ v_p \hbox{~is the $p$-evaluation};
\\
 N_{\gamma,n}(a)&= \sharp \{ r\in (\ZZ/a\ZZ)^2~|\frac{1}{2}(r-\gamma)^2+n\equiv 0 \mod a \};\\
 L_{\gamma, n}(k,p)&=(1-p^{1-k}) \sum\limits_{v=0}^{\omega_p-1} N_{\gamma,n}(p^v)p^{-kv}+N_{\gamma,n}(p^{\omega_p})p^{-k\omega_p}.\\
\end{align*}
When $k=5$,  $$L(5,\chi_{-3})=\frac{2^4\pi^5}{5! \sqrt{3}
}\sum\limits_{n=1}^{3}\chi_{-3}(n)
B_5(1-n/3)=\frac{2^5\pi^5}{3^5\Gamma(5)\sqrt{3}},$$

\noindent where $B_k(x)=\sum\limits_{k=0}^n {{n}\choose{r}} B_{n-k}x^k$ is the Bernoulli polynomial.
  Thus one obtains that
\begin{align*}
\overrightarrow{E_5}(q)=&2v_0+ \sum\limits_{i=0}^{2} \sum\limits_{n\in \ZZ+\frac{1}{3}i^{2}\atop n > 0} 486n^4 \prod\limits_{p|18n}
\frac{L_{\gamma_i,n}(5,p)}{1-\chi_{-3}(p)p^{-5}}q^n v_i\\ =&(2+492q+7200q^2+39372q^3+\ldots) v_0+(6q^{1/3}+1446q^{4/3}+\\ &14412q^{7/3}+\ldots)v_1+(6q^{1/3}+1446q^{4/3}+14412q^{7/3}+\ldots)v_2 .
\end{align*}

\subsubsection{Rankin-Cohen bracket}
Given any two level $N$ scalar-valued modular forms $f(q),g(q)$ on the upper half plane $\cH$ of weight $k_1$ and $k_2$. The n-th Rankin-Cohen bracket is defined as follows:
$$[f(q),g(q)]_n=\sum\limits_{r=0}^n (-1)^r \left(\begin{array}{c}n+k_1-1 \\n-r\end{array}\right)\left(\begin{array}{c} n+k_2-1\\r\end{array}\right) f^{(r)}(q)\cdot g^{(n-r)}(q)$$
where $ f^{(r)}$ denotes the r-th differential of $f$ with respect to $\tau$.

For a vector-valued modular form $$\overrightarrow{F}(q)=\sum\limits_{\gamma\in M^\vee/M} F_\gamma v_\gamma\in \Mod(\Mp(\ZZ), k_1,\rho_M),$$
 one can extend the Rankin-Cohen bracket to $\overrightarrow{F}(q)$ and $g(q)$ as follows,
\begin{equation}
[\overrightarrow{F}(q),g(q)]_n=\sum\limits_{\gamma\in M^\vee/M}[F_\gamma(q),g(q)]_n v_\gamma.
 \end{equation}

According to  \cite[Lemma 5]{MP07},  we have the following result:
\begin{lemma}\label{lem2} The vector-valued functions
\begin{equation}\label{eq5.7}
\overrightarrow{F_n}(q):=[\overrightarrow{E_5}(q), E_{6-2n}(q)]_n,n=0,1.
\end{equation}
are vector-valued modular forms of weight $11$ and type $\rho_{\Lambda_0}$.
\end{lemma}

\subsection{Expression of the generating series}
Now we are ready to give an explicit expression of  $\overrightarrow{\Theta}(q)$.  From the dimension formula of Bruinier in \cite{Br02a},  we know that
\begin{equation}\label{eq3.3}
\dim \Mod(\Mp(\ZZ),11,\rho_{\Lambda_0})=2.
\end{equation}

It follows that
\begin{theorem} \label{thm1}
Let $\overrightarrow{F_0}(q),\overrightarrow{F_1}(q)$ be the
vector-valued modular forms constructed in Lemma \ref{lem2}.
Then $\{\overrightarrow{F_0}(q),\overrightarrow{F_1}(q) \}$
 is a basis of $\Mod(\Mp(\ZZ), 11,\rho_{\Lambda_0})$ and
 \begin{equation}\label{eq3.4}
 \begin{aligned}
 \overrightarrow{\Theta}(q)&=-\overrightarrow{F_0}(q)-\frac{3}{4}\overrightarrow{F_1}(q)\\ &= (-2+192 q+196272q^2+\ldots)v_0+(0+3402 q^{4/3}+\\~&917568q^{7/3}+\ldots)v_1+(3402 q^{4/3}+917568q^{7/3}+\ldots)v_2.
 \end{aligned}
 \end{equation}
 \end{theorem}
\begin{proof} By checking the coefficients of the term $q^0v_0$,  we know that $\overrightarrow{F_0}$  and $\overrightarrow{F_1}$ are linearly independent.
Thus they form a basis of $\Mod(\Mp(\ZZ), 11,\rho_{A_2})$ by dimension considerations.
To obtain the expression (\ref{eq3.4}),  it suffices to use the
following two constraints:
\begin{enumerate1}
 \item The degree of the Hodge bundle
$R^3\pi_\ast(\Omega^1_{\pi}) $ is $-2$, which gives the coefficient
of $q^0v_0$. By Grothendieck-Riemann-Roch, we have the following
Chern character computation:
\begin{align*}\hbox{ch}(\pi_!
\Omega^1_{\fX/\PP^1})&=\hbox{ch}(-R^1\pi_\ast
\Omega_{\fX/\PP^1}^1-R^3\pi_\ast\Omega^1_{\fX/\PP^1}) \\&=
\pi_\ast(\hbox{ch}(\Omega^1_{\fX/\PP^1})\hbox{td}(\hbox{T}_{\fX/\PP^1}))\\
&=-2+2c_1(\cO_{\PP^1} (1)),
\end{align*} where
$\hbox{T}_{\fX/\PP^1}$ is the relative tangent bundle. Since $R^1\pi_\ast(\Omega^1_{\fX/\PP^1})=R^2\pi_\ast
\CC$ is trivial by the Lefschetz hyperplane theorem, we get
$\deg(\Omega^1_{\fX/\PP^1})=-2$.
\item The coefficient of $q^{1/3}v_1$ is $N_2=0$ as shown in $\S$2.3.
\end{enumerate1}
\end{proof}
\end{section}

\section{Proof of Theorem 1}
To prove our main theorem, we first start with
a Lemma on the modularity on the components of a vector-valued modular form:
\begin{lemma} \label{lem3}Let $\overrightarrow{F}=\sum\limits_{i=0}^2 F_iv_i$ be an element in $\Mod(\Mp(\ZZ),k,\rho_{\Lambda_0})$.
 Then the following are true:
\begin{enumeratea}
 \item $F_0$ is a scalar-valued modular form for
$\Gamma_0(3)$ of weight $k$ with character $\chi$.
 \item $F_1=F_2$ is a scalar-valued modular form for
$\Gamma_1(3)$ of weight $k$ with character  $\chi'(A)=e^{\frac{2b\pi
i}{3}}$, where $$\Gamma_1(3)=\left\{ \left(\begin{array}{cc}a & b
\\c & d\end{array}\right) \in SL_2(\ZZ) |~a\equiv d\equiv 1,
~c\equiv 0\mod 3\right\}.$$
 \item The sum $\sum\limits_{i=0}^2 F_i$ is a
scalar-valued modular form for $\Gamma^1(3)$ of weight $k$ with
character $\chi$.
\end{enumeratea}
\end{lemma}
\begin{proof}  Statement (i) and (ii)  follows from \cite{Sc12} $\S$2. For (iii), since the generators of the congruence subgroup $\Gamma^0(3)$ are
 \begin{equation}\label{eq4.1} \left(\begin{array}{cc}1 & 0 \\-1 & 1\end{array}\right) \textrm{ and }~~~~~~~~~~~~~~~~~~~~~\left(\begin{array}{cc}2 & 3 \\-1 & -1\end{array}\right),\end{equation}
 then statement (iii) follow from a direct computation by checking the modularity on these generators.
\end{proof}
\noindent{\bf Proof of Theorem 1}.
Let us write $\overrightarrow{\Theta}(q)=\sum\limits_{i=0}^2\Theta_iv_i\in \Mod(\Mp(\ZZ),11,\rho_{\Lambda_0})$. Then it is not difficult to see that \begin{equation}\label{eq4.2}
\Theta(q)=\Theta_0(q)+\Theta_1(q)=\frac{1}{2}(\Theta_0(q)+\sum_{i=0}^2\Theta_i(q)),
\end{equation}
from the expressions \eqref{eq2.5} and \eqref{eq3.1}.

Next, let  $\Mod(\Gamma_0(3), \chi)$ (resp.~$\Mod(\Gamma_0(3),
\chi)$) denote the space of scalar-valued modular forms  with
character $\chi$ for $\Gamma_0(3)$ (resp.~$\Gamma^0(3)$).
 It is known by \cite{Bo00} $\S$12 that  $\Mod(\Gamma_0(3), \chi)$ is a polynomial ring generated by two modular forms of weight $1$ and  weight $3$.

 By Proposition \ref{prop1},  we thus get that $\Mod(\Gamma_0(3), \chi)$ is generated by $\alpha(q)$ of weight $1$ and $\beta(q)$ of weight $3$.
 Since $\Theta_0\in \Mod(\Gamma_0(3), \chi)$ has weight $11$ by Lemma \ref{lem3},  $\Theta_0$ can be expressed as a linear combination
of $$\alpha^{11}(q),\beta^{8}(q)\beta(q), \alpha^5(q)\beta^2(q), \alpha^2(q)\beta^3(q). $$
 The coefficients computation shows that
 \begin{equation}\label{eq4.3}
\Theta_0(q) = -2\alpha^{11}+324 \alpha^8\beta +183708\alpha^5\beta^2  +4408992 \alpha^2\beta^3.
\end{equation}
 Similarly,  $\Mod(\Gamma^0(3),\chi)$ is a polynomial ring generated by $\alpha(q^{\frac{1}{3}})$ and $\beta(q^{\frac{1}{3}})$.
 Then the modular form $\sum\limits_{i=0}\Theta_i\in \Mod(\Gamma^0(3),\chi)$ has weight $11$ and can be expressed as
 \begin{equation}\label{eq4.4}
\begin{aligned}
\sum_{i=0}^2\Theta_i=&-2\alpha^{11}(q^{\frac{1}{3}})+132\alpha^8(q^{\frac{1}{3}})\beta(q^{\frac{1}{3}})- 2772\alpha^5(q^{\frac{1}{3}})\beta^2(q^{\frac{1}{3}})\\ &+18144\alpha^2(q^{\frac{1}{3}})\beta^3(q^{\frac{1}{3}}).
 \end{aligned}
 \end{equation}
 Our main theorem follows from \eqref{eq4.2}, \eqref{eq4.3} and \eqref{eq4.4}. \qed


\bibliographystyle {plain}

\bibliography{MFSCF}

\end{document}